\documentclass[11pt]{amsart}
\usepackage[alphabetic]{amsrefs}
\usepackage[sc, osf]{mathpazo}
\usepackage[T1]{fontenc}
\usepackage{hyperref}
\hypersetup{colorlinks = true, linkcolor = {blue}, urlcolor  = black}
\usepackage[all,2cell]{xy} \UseAllTwocells \SilentMatrices \SelectTips{cm}{}
\usepackage{latexsym,amsfonts,amssymb}
\usepackage{amsmath,amsthm,amscd}
\usepackage[dvips]{epsfig}
\usepackage{psfrag}

\usepackage{color}
\usepackage{etoolbox}
\usepackage{graphicx}
\usepackage{a4wide}
\usepackage[margin=1in]{geometry}
\usepackage{graphicx}
\usepackage{fancyhdr}
\usepackage{mathtools}
\usepackage{topcapt}
\usepackage{booktabs}
\usepackage{comment}
\usepackage[scaled]{helvet}
 \usepackage[colorinlistoftodos]{todonotes}
\usepackage[T1]{fontenc}
\input xy
\xyoption{all}

\newcommand{\ot}{\otimes}

\newcommand{\C}{\mathbb{C}}

\newcommand{\inv}{^{-1}}

\newcommand{\Mat}{\text{\rm Mat}}
\newcommand{\GL}{\text{\rm GL}}

\renewcommand{\O}{\mathcal{O}}

\newcommand{\U}{\mathcal{U}}

\newcommand{\Spec}{\text{\rm Spec}}
\newcommand{\Specm}{\text{\rm Specm}}

\newcommand{\id}{\text{\rm Id}}

\newcommand{\DqN}{\mathcal{D}_q(\C^N)}
\newcommand{\DqNm}{\mathcal{D}_q(\C^{N-1})}
 
\newcommand{\Dq}{\mathcal{D}_q}

    \newcommand{\sD}{\mathcal{D}}
    \newcommand{\D}{\mathcal{D}}

\newcommand{\paren}[1]{\ensuremath{\left(#1\right)}}

\newcommand{\detq}{\text{\rm det}_q}

\newcommand{\OqGL}{\O_q(\GL_2)}
\newcommand{\OqGLell}{\O_q^{(\ell)}(\GL_2)}
\newcommand{\g}{\mathfrak g}
\newcommand{\gl}{\mathfrak{gl}}

\theoremstyle{plain}
\newtheorem{theorem}{Theorem}[section]
\newtheorem{prop}[theorem]{Proposition}
\newtheorem{lemma}[theorem]{Lemma}
\newtheorem{cor}[theorem]{Corollary}

\theoremstyle{definition}
\newtheorem{definition}[theorem]{Definition}
\newtheorem{example}[theorem]{Example}
\newtheorem{rmk}[theorem]{Remark}
\numberwithin{equation}{section}

\begin{document} 

\title{Quantum Weyl algebras and reflection equation algebras at a root of unity}
\author{ Nicholas Cooney \qquad Iordan Ganev \qquad David Jordan}
\date{}
\parindent = 10pt
\parskip = 0pt

\begin{abstract} 
We compute the center and Azumaya locus in the simplest non-abelian examples of quantized multiplicative quiver varieties at a root of unity:  quantum Weyl algebras of rank $N$, and quantum differential operators on the quantum group $\GL_2$.  These examples illustrate in elementary terms much more general phenomena explored further in \cite{GanevquantumFrobeniuscharacter2019}.  
\end{abstract}

\maketitle

\section{Introduction}\label{sec:intro}

Many algebras appearing in representation theory exhibit interesting $q$-deformations -- flat families of algebras depending on a nonzero complex parameter $q$, whose value at $q=1$ is the original algebra. Perhaps the most prominent example is the quantum group $U_q(\g)$ attached to a Lie algebra $\g$, which may be regarded as deforming the universal enveloping algebra $U(\g)$, or alternatively the big Bruhat cell $BB^-$ in $G$. Other examples -- which typically  pivot crucially  on the quantum group --  include Hecke algebras, quantum coordinate algebras, quantum Weyl algebras, and quantized multiplicative quiver and hypertoric varieties. 

These $q$-deformations tend  to exhibit a common structure for generic $q$, and develop radically different structures when $q$ is a root of unity.  In particular, such specializations are a reliable source of interesting \emph{Azumaya} algebras: an algebra over $\C$ is called Azumaya if it is finitely generated and projective as a module over its center, and its fiber at any point in the spectrum of the center is a matrix algebra. Such properties have been observed for many different algebras, including Fadeev--Reshetikhin--Turaev algebras, quantum Weyl algebras, double affine Hecke algebras, and quantized multiplicative quiver and hypertoric varieties  \cite{BackelinLocalizationquantumgroups2008, BezrukavnikovCherednikalgebrasHilbert2006, CooneyQuantummultiplicativehypertoric2016a, GanevQuantizationsmultiplicativehypertoric2018, VaragnoloDoubleaffineHecke2010, LevittQuantizedWeylalgebras2018, DeconciniQuantumFunctionAlgebra1994, BrownramificationscentresQuantised2002}.   

We consider the simplest non-abelian examples of quantum multiplicative quiver varieties: the algebra $\DqN$ of $q$-difference operators on $\C^N$, a quantum Weyl algebra in the sense of \cite{GiaquintoQuantumWeylAlgebras1995}, and the algebra $\Dq(\GL_2)$ of $q$-difference operators on $\GL_2$.  This paper is intended as a supplement to \cite{GanevquantumFrobeniuscharacter2019}, where a much wider class of $q$-deformations are considered.  The methods in \cite{GanevquantumFrobeniuscharacter2019} emerged as an abstraction of the techniques applied to the basic examples analyzed in the current paper. 

In the case of the quantum Weyl algebra, for $q$ a root of unity, we identify the center of $\DqN$ with the coordinate algebra of $T^*\C^N$, equipped with the non-standard Poisson bivector,
$$ \pi = \sum_{j>i} \left( y_j y_i  \frac{\partial}{\partial y_j}  \wedge  \frac{\partial}{\partial y_i} - z_j z_i  \frac{\partial}{\partial z_j}  \wedge  \frac{\partial}{\partial z_i} \right) + \sum_{i \neq j} y_i z_j   \frac{\partial}{\partial y_i}  \wedge  \frac{\partial}{\partial z_j}  + 2 \sum_{i=1}^N \left( 1 + \sum_{k=1}^i y_k z_k  \right)  \frac{\partial}{\partial y_i} \wedge  \frac{\partial}{\partial z_i} ,$$
where $\{y_i\}$ and $\{z_i\}$ are the base and fiber coordinates on  $T^*\C^N$. We consider the non-degeneracy locus $(T^*\C^N)^\circ$ of $\pi$,  define a corresponding non-commutative localization $\DqN^\circ$  of $\DqN$ at certain grading operators (analogs of `Euler operators'), and show the following:

\begin{theorem}\label{thm:introDqN}(Theorem \ref{thm:DqNAzumaya}) Suppose $q$ is a primitive $\ell$-th root of unity, where $\ell > 1$ is odd. Then the algebra $\DqN^\circ$ is Azumaya over $(T^*\C^N )^\circ$, and its rank is $\ell^N \times \ell^N$. \end{theorem}

As our second example, we specialize to the case $N=2$ and consider the  reflection equation algebra $\O_q(\GL_2)$, as well as its `Heisenberg double' $\Dq(\GL_2)$ of $q$-difference operators on the group $\GL_2$.  The reflection equation algebra is a $q$-deformation of the coordinate algebra of the general linear group, and it quantizes the Semenov-Tian-Shansky Poisson bracket \cite{MudrovquantizationSemenovTianShanskyPoisson2007}.  The algebra $\D_q(\GL_2)$ is then a $q$-deformation of $\O(\GL_2\times \GL_2)$.  We note that the Semenov-Tian-Shansky bracket on $\GL_2 \times \GL_2$ is generically symplectic  \cite{Semenov-Tian-ShanskyPoissonLieGroups1994}. The following claim is a special case of \cite[Proposition 2.2.3]{VaragnoloDoubleaffineHecke2010}.

\begin{prop}\label{prop:GL2ellintro} (Corollary \ref{cor:GL2ell})  Suppose $q$ is a primitive $\ell$-th root of unity, where $\ell >1$ is odd.  Then:
\begin{enumerate}
\item We have a central embedding of the commutative algebra $\O(\GL_2)$ into $\O_q(\GL_2)$.
\item We have a central embedding of $\O(\GL_2\times\GL_2)$ into $\Dq(\GL_2)$.
\end{enumerate}
\end{prop}

Each of the spaces $(T^*\C^2)^\circ$ and $\GL_2\times \GL_2$ carry a non-degenerate Hamiltonian $\GL_2$-Poisson structure.  In other words, there are `group-valued moment maps'
$$\tilde \mu : (T^*\C^2)^\circ \rightarrow \GL_2,\qquad \tilde\phi: \GL_2\times \GL_2\to \GL_2,$$
which are $\GL_2$-equivariant and describe the $\GL_2$-action by Hamiltonian flows \cite{AlekseevLiegroupvalued1998}.  The non-degeneracy condition implies in particular that the preimage $\tilde\mu^{-1}(BB^-)$ of the big Bruhat cell in each case is an open symplectic leaf.  The central embedding of $\O(T^*\C^2)$ into $\Dq(\C^2)$ is compatible with the moment maps in the following sense:

\begin{theorem}  Suppose $q$ is a primitive $\ell$-th root of unity, where $\ell >1$ is odd.  The central embedding of $\O(\GL_2)$ into $\O_q(\GL_2)$ fits into the following commutative diagram:
\[ \xymatrix{  \O_q(\GL_2 ) \ar[rr]^{\mu_q} & & \Dq(\C^2)^\circ \\
		\O(\GL_2 ) \ar@{^{(}->}[u] \ar[rr]^{\tilde \mu^\#  \ \ } & & \O( (T^* \C^2 )^\circ ) \ar@{^{(}->}[u] }  \]
Moreover, the Azumaya locus $(T^*\C^2)^\circ$ is the preimage under $\tilde \mu$ of the big Bruhat cell of $\GL_2$.  \end{theorem}

In \cite{GanevquantumFrobeniuscharacter2019}, such a commutative diagram is called a Frobenius quantum moment map.  These exist also in the case of $\Dq(\GL_N)$, as was originally proven in \cite{VaragnoloDoubleaffineHecke2010}. In Proposition \ref{prop:dqgl2mmcenter}, we describe the Frobenius quantum moment map in this case in coordinates for $N=2$, where already the explicit formulas which arise are interesting and somewhat unexpected.\\

\paragraph{\bf Motivation} The algebras $\Dq(\C^N)$ and $\Dq(\GL_2)$ are basic examples of framed quantum multiplicative quiver varieties \cite{JordanQuantizedmultiplicativequiver2014}.  Their quantum Hamiltonian reductions therefore define quantizations of multiplicative quiver varieties, and in \cite{GanevquantumFrobeniuscharacter2019} a general framework is described for establishing the Azumaya locus for both the unframed and framed quantizations.  Our aim in this short paper is to give an elementary proof in these basic examples, as motivation for the machinery needed for the general case. \\

\paragraph{\bf Acknowledgments. } The authors would like to thank  Kobi Kremnitzer for his encouragement and advice. I.G. is grateful for the support of the Advanced Grant ``Arithmetic and Physics of Higgs moduli spaces'' No.\ 320593 of the European Research Council. The work of D.J. was supported by the European Research Council under the European Union's Horizon 2020 research and innovation programme [grant agreement no.\ 637618].


\section{An algebra of $q$-difference operators}\label{sec:dqn}

\subsection{Basic definitions}

\begin{definition} Let $q \in \C^\times$ and let $N$ be a non-negative integer. The algebra $\DqN$ of $q$-difference operators on $\C^N$ is the quotient of the free algebra on the generators $x_1, \dots, x_N, \partial_1, \dots, \partial_N$ by the relations: 
	$$x_j x_i = q x_i x_j \qquad \text{for $j >i$}$$ 
	$$\partial_j \partial_i = q\inv \partial_i \partial_j \qquad \text{for $j >i$}$$ 
	$$\partial_j x_i = q x_i \partial_j \qquad \text{for $j \neq i$}$$ 
	$$\partial_i x_i = q^2 x_i \partial_i + (q^2-1)(1 + \sum_{j=1}^{i-1} x_j \partial_j)$$ 
\end{definition}

\begin{definition}  The quantum coordinate algebra $\O_q(\C^N)$ of affine $N$-space is defined as the quotient of the free algebra on generators $x_1, \dots, x_N$ by the relation $x_j x_i = q x_i x_j$ for $j >i$. As a vector space, it is spanned by the monomials $(x_1)^{n_1} \cdots (x_N)^{n_N}$, where  $n_i \geq 0$ for all $i = 1, \dots, N$.  There is an action of $\DqN$ on $\O_q(\C^N)$ given by
$$x_i \triangleright f = x_i f $$
$$ \partial_i \triangleright [(x_1)^{n_1} \cdots (x_N)^{n_N}] = (q x_1)^{n_1} \cdots (qx_{i-1})^{n_{i-1}} \frac{(q x_i)^{n_i} - x_i^{n_i}}{x_i} x_{i+1}^{n_{i+1}} \cdots (x_N)^{n_N},$$
for $f \in \O_q(\C^N)$. See \cite[12.3.3]{KlimykQuantumGroupsTheir1997} for more details and examples. \end{definition}

Inspection of the defining relations of $\DqN$ gives the following lemma (where we define $\D_q(\C^0) := \C$).

\begin{lemma}\label{lemma:DqNbasics} Let $q \in\C^\times$.
	\begin{enumerate} 
		\item The algebra $\DqN$ is spanned as a vector space by the elements $x_1^{m_1} \cdots x_N^{m_N} \partial_1^{n_1} \cdots \partial_N^{n_N}$, where the $m_i, n_i$ are nonnegative integers.  
		\item For $N \geq 1$, the subalgebra of $\DqN$ generated by the elements $x_1, \dots, x_{N-1}, \partial_1, \dots, \partial_{N-1}$ is isomorphic to $\DqNm$. 
	\end{enumerate}
\end{lemma}

Henceforth, we identify $\DqNm$ as a subalgebra of $\DqN$. Setting $q=1$ in the definition of $\DqN$, we obtain the algebra $\O(T^*  \C^N)$ of functions on the cotangent bundle $T^*\C^N $, and Lemma \ref{lemma:DqNbasics} shows that $\DqN$ is a flat deformation of $\O(T^* \C^N )$. On the other hand, replacing each $\partial_i$ by $\frac{\partial_i}{q^2 -1}$ and then setting $q=1$, we obtain a variant of the $N$th Weyl algebra of differential operators on affine $N$ space. In this variant, all generators commute, with the exception of $x_i$ and $\partial_i$, which satisfy the relation $$\partial_i x_i = 1 + x_1 \partial_1 + \dots + x_i \partial_i.$$ Hence, $\DqN$ is $q$-deformation of both the algebra $\O(T^*\C^N)$ and a variant of the Weyl algebra.

\begin{rmk} The algebra $\DqN$ is associated to the $A_2$ quiver $\bullet \rightarrow \bullet$ equipped with the dimension vector $\paren{1,N}$. It  is isomorphic to the algebra $\D_q(\Mat_{(1,N)}(Q))$ defined in \cite[Example 4.8]{JordanQuantizedmultiplicativequiver2014} via the bijection $x_i \leftrightarrow a^{1}_{N-i+1}$ and $\partial_i \leftrightarrow \partial_1^{N-i+1}$. It is also an example of the quantum Weyl algebras of \cite{GiaquintoQuantumWeylAlgebras1995}. \end{rmk}

\begin{definition} For $i = 0, 1, \dots, N$, set $\beta_i = 1+ \sum_{j =1}^{i} x_j \partial_j \in \DqN.$   \end{definition}

Direct computations yield the following result,  which shows that the elements $\beta_i$ have nice $q$-commutation properties with the generators, and can be thought of as ``partial grading operators'' or Euler operators. 

\begin{lemma}\label{lem:betacomm} Let $i,j \in \{1, \dots, N\}$. The elements $\beta_i$ and $\beta_j$ commute, and for $n \geq 1$, the following identities hold in $\DqN$: 
	$$ \partial_i x_i^n = q^{2n} x_i^n \partial_i + (q^{2n} -1)\beta_{i-1} x_i^{n-1}, \qquad \text{ and} \qquad \partial_i^n x_i = q^{2n} x_i \partial_i^n + (q^{2n} -1)\beta_{i-1} \partial_i^{n-1}.$$
	If $j > i$, then $ \beta_i x_j = x_j \beta_i  \ \text{and} \ \beta_i \partial_j = \partial_j \beta_i,$ while if $j \leq i$, then $ \beta_i x_j = q^2 x_j \beta_i \ \text{and} \ \beta_i \partial_j = q^{-2} \partial_j \beta_i.$ \end{lemma}


\subsection{The center} We fix $q$ to be a primitive $\ell$-th root of unity in $\C$, where $\ell >1$ is odd. Let $Z_N$ denote the center of $\DqN$.  

\begin{prop}\label{prop:center} The center $Z_N$ is the subalgebra generated by the elements $x_1^\ell$, $\dots$, $x_N^{\ell}$, $\partial_1^\ell$, $\dots, \partial_N^\ell$, and is isomorphic to a polynomial algebra:	$$Z_N = \C[x_1^\ell, \dots, x_N^\ell, \partial_1^\ell, \dots, \partial_N^\ell].$$
\end{prop}

\begin{proof} We proceed by induction on $N$. The base case $N=0$ is clear. Fix $N\geq 1$, and let $Z_\ell$ denote the subalgebra of $\DqN$ generated by the elements $x_1^\ell, \dots, x_N^\ell, \partial_1^\ell, \dots, \partial_N^\ell$. (The algebra $Z_\ell$ is often referred to as the $\ell$-center.) By Lemma \ref{lem:betacomm} and the defining relations of $\DqN$, the algebra $Z_\ell$ is contained in the center $Z_N$ of $\DqN$ and is isomorphic to the polynomial algebra in variables  $x_1^\ell, \dots, x_N^\ell, \partial_1^\ell, \dots, \partial_N^\ell$. Thus, it suffices to show that $Z_N$ is contained in $Z_\ell$.  To this end, let $z \in Z_N$. By Lemma \ref{lemma:DqNbasics}, we can write $$z = \sum_{m,n} z^\prime_{m,n}x_N^m \partial_N^m,$$ where the sum ranges over nonnegative integers $m,n$ and $z^\prime_{m,n}$ belongs to $\DqNm$. The inclusion $Z_N \subseteq Z_\ell$ (and hence the proposition) follows from the following claim, whose proof is divided into three steps:\\
	
\noindent {\bf Claim.} The element $z^\prime_{m,n}$ belongs to $Z_\ell$ if $\ell$ divides both $m$ and $n$. Otherwise, $z^\prime_{m,n}=0.$ \\
	
{\bf Step 1.} We argue that $x_N$ and $\partial_N$ commute with each $z^\prime_{m,n}$. First note that $\beta_{N-1}$ commutes with each $z^\prime_{m,n}$. This fact is a consequence of the observations that $z$ is central and that $\beta_{N-1}$ commutes with $x_N$ and $\partial_N$. Next, fix $m,n$ and abbreviate $z^\prime_{m,n}$ by $z^\prime$. Write 
$$z^\prime=\sum_{\underline{m},\underline{n}}c_{\underline{m},\underline{n}} x^{\underline{m}}\partial^{\underline{n}},$$
where the sum ranges over pairs of $(N-1)$-tuples $\underline{m} = (m_1, \dots, m_{N-1})$ and $\underline{n} = (n_1, \dots, n_{N-1})$, and $ x^{\underline{m}}\partial^{\underline{n}}$ abbreviates the element $x_1^{m_1} \cdots x_{N-1}^{m_{N-1}}\partial_1^{n_1} \cdots \partial_{N-1}^{n_{N-1}}.$ Since $\beta_{N-1}$ commutes with $z^\prime$, we use Lemma \ref{lem:betacomm} to obtain 
$$\sum_{\underline{m},\underline{n}}c_{\underline{m},\underline{n}} x^{\underline{m}}\partial^{\underline{n}}\beta_{N-1} = z^\prime \beta_{N-1} = \beta_{N-1} z^\prime = \sum_{\underline{m},\underline{n}} q^{\left(2 \sum_{j=1}^{N-1} (m_j - n_j)\right)} c_{\underline{m},\underline{n}} x^{\underline{m}}\partial^{\underline{n}}\beta_{N-1}.$$
It follows that $c_{\underline{m},\underline{n}}= 0$ whenever $\ell$ does not divide the sum $\sum_{j=1}^{N-1} (m_j - n_j)$. (We have used the hypothesis that $\ell$ is odd.) Consequently, using the defining relations of $\DqN$, we conclude that
$$x_N z^\prime = \sum_{\underline{m},\underline{n}}c_{\underline{m},\underline{n}} x_N x^{\underline{m}}\partial^{\underline{n}} = \sum_{\underline{m},\underline{n}} q^{ \sum_{j=1}^{N-1} (m_j - n_j)} c_{\underline{m},\underline{n}} x^{\underline{m}}\partial^{\underline{n}}\beta_{N-1} = z^\prime x_N.$$
An analogous argument shows that $\partial_N z^\prime = z^\prime \partial_N$. \\
	
{\bf Step 2.} We argue that $z^\prime_{m,n} =0$ if $\ell$ does not divide both $m$ and $n$. We have that
	\begin{align*}
		x_N z &= \sum_{m,n} z^\prime_{m,n} x_N^{m+1} \partial_N^n = \sum_{m \geq 1; n\geq 0} z^\prime_{m-1,n} x_N^{m} \partial_N^n, \qquad \text{and} \\
		z x_N &= \sum_{m,n} q^{2n}z^\prime_{m,n} x_N^{m+1} \partial_N^n + (q^{2n} -1 )\beta_{N-1} z^\prime_{m,n} x_N^{m} \partial_N^{n-1} \\ &= \sum_{m\geq 1; n\geq 0} q^{2n}z^\prime_{m-1,n} x_N^{m} \partial_N^n + \sum_{m,n}(q^{2(n+1)} -1 )\beta_{N-1} z^\prime_{m,n+1} x_N^{m} \partial_N^{n}\end{align*}
	Since $x_N z = z x_N$, we deduce that\footnote{One also deduces that $(q^{2(n+1)}-1)\beta_{N-1}z^\prime_{0, n+1} = 0$ for any $n \geq 0$, but we do not need this.}
	$$z^\prime_{m-1, n} = q^{2n} z^\prime_{m-1,n} + (q^{2(n+1)} - 1) \beta_{N-1} z^\prime_{m, n+1}$$
	for $m \geq 1$ and $n \geq 0$. Rearranging this equation, it is straightforward to verify that this identity gives
	$$(q^{2n} -1) z^\prime_{m,n} = (-1)^i (q^{2(n+i)} -1) \beta_{N-1}^i z^\prime_{m+i, n+i}$$
	for $m,n,i \geq 0$. Suppose that $\ell$ does not divide $n$, so that $q^{2n} -1 \neq 0$. There is a unique $i \in \{1, \dots, \ell-1\}$ such that $\ell$ divides $n+i$. For this choice of $i$, we have 
	$$z^\prime_{m,n} = \frac{(-1)^i (q^{2(n+i)} -1) \beta_{N-1}^i z^\prime_{m+i, n+i}}{q^{2n} -1} =  \frac{(-1)^i (0) \beta_{N-1}^i z^\prime_{m+i, n+i}}{q^{2n} -1} = 0.$$
	An analogous argument using the identity $\partial_N z = z\partial_N$ shows that $z^\prime_{m,n} = 0$ whenever $\ell$ does not divide $m$. \\
	
	{\bf Step 3.} Finally, we show that if $\ell$ divides both $m$ and $n$, then $z^\prime_{m,n}$ belongs to $Z_\ell$. Since $z_{{m},{n}}^{\prime}$ lies in the subalgebra $\sD_{N-1}\paren{Q}$ it suffices, by induction, to show that it lies in the center. This follows from the fact that $x_{j}$ and $\partial_{j}$ commute with $z$ for $1\leq j\leq N-1$ and that $\ell$ divides both $m$ and  $n$ whenever $z_{{m},{n}}^{\prime}$ is nonzero.
\end{proof}

\begin{prop}\label{prop:betaeqns} The following formulas hold in $\DqN$, for $i = 0, \dots, N$: $$\beta_i^\ell = 1 + \sum_{j=1}^{i-1} x_j^\ell \partial_j^\ell.$$ \end{prop}

\begin{proof} We proceed by induction on $N$. The base case $N=0$ is clear. Fix $N \geq 1$. The inclusion of $\DqNm$ in $\DqN$ implies $\beta_i^\ell$ satisfies the formula for $i = 0, \dots, N-1$. We claim that the following identity holds in $\DqN$, for any $n \geq1$:
	$$(x_N \partial_N + \beta_{N-1})^n = q^{n(n-1)} x_N^n \partial_N^n + \sum_{k=1}^{n-1} c_k^{(n)} x_N^k \partial_N^k + \beta_{N-1}^n,$$
	where $c_k^{(n)}$ are certain elements of $\DqNm$. The claim is easily verified by induction, using the following facts: 
	\begin{enumerate}
		\item The element $\beta_{N-1}$ commutes with each of $x_N$ and $\partial_N$. 
		\item For any $k \geq 0$, we have $x_N^k \partial_N^k x_N \partial_N = q^{2k} x_N^{k+1} \partial_N^{k+1} + (q^{2k} -1) x_N^k \partial_N^k.$
	\end{enumerate}
	
Specializing the above formula for $n = \ell$, rearranging terms, and noting that $\beta_N = x_N \partial_N + \beta_{N-1}$, we obtain  $$\sum_{k=1}^{\ell-1} c_k^{(\ell)} x_N^k \partial_N^k = \beta_{N}^\ell - x_N^\ell \partial_N^\ell - \beta_{N-1}^\ell.$$ The right-hand side is central, so it follows that the left-hand side is also central. By Proposition \ref{prop:center}, we conclude that $c_k^{(\ell)} = 0$ for $k = 1, \dots, \ell-1$. Therefore, $$\beta_N^\ell =  x_N^\ell \partial_N^\ell + \beta_{N-1}^\ell = 1 + \sum_{j=1}^{N} x_j^\ell \partial_j^\ell.$$ \end{proof}

\begin{rmk} The elements $c_k^{(n)}$  satisfy the recursive formula $c_k^{(n)} = q^{2k} c_k^{(n-1)} \beta_{N-1} + q^{2(k-1)} c_{k-1}^{(n-1)}$ and  thus can be easily related  to quantum factorials.\end{rmk}


\subsection{The Azumaya locus} We include a few generalities on Azumaya algebras. Let $A$ be an algebra over $\C$ and let $Z = Z(A)$ be its center. 

\begin{definition} Suppose $A$ is finitely generated and projective as a $Z$-module. The Azumaya locus of $A$ is the subset of the maximal spectrum of $Z$ consisting of points  $\mathfrak m \in \Specm(Z)$ such that there exists $n >0$ and an isomorphism
	$$A \ot_Z Z /\mathfrak m \simeq \Mat_{n \times n} (\C).$$ If the Azumaya locus of $A$ is all of $\Specm(Z)$, we say that $A$ is Azumaya over $Z$. \end{definition}

In other words, if we regard $A$ as defining a sheaf of algebras over $\Spec(Z)$, then $A$ is Azumaya if this sheaf is \'etale-locally the endomorphisms of a vector bundle on $\Spec(Z)$.

We now return to the setting of difference operators, and continue to fix $q$ to be a primitive $\ell$-th root of unity in $\C$, where $\ell >1$ is odd.  We identify $\Spec(Z_N)$ with the cotangent bundle $T^* \C^N$ via Proposition \ref{prop:center}, and regard $\DqN$ as a coherent sheaf of algebras on $T^*\C^N$. 

\begin{definition} Define $Z_N^\circ = Z_N[\{(\beta_i^\ell)\inv \}]$ to be the localization of $Z_N$ along the multiplicative set generated by the $\beta_i^\ell$. Define $\DqN^\circ = \DqN \otimes_{Z_N} Z_N^\circ$.  \end{definition}

Using Propositions \ref{prop:center} and \ref{prop:betaeqns}, we identify the closed points of $\Spec(Z_N^\circ)$ with the subset 
$$\{\nu = (\nu_j, \check \nu_j)_{j=1}^N \in T^* \C^N \ | \ 1 +  \sum_{j=1}^{i-1} \nu_j \check \nu_j \neq 0 \text{ for all $i$} \} \subseteq T^* \C^N.$$

\begin{lemma} The multiplicative set generated by the $\beta_i$ satisfies the Ore condition and $\DqN^\circ$ coincides with the localization $\DqN [ \{ \beta_i\inv \} ].$ \end{lemma}

\begin{proof} The result follows immediately from the $q$-commutativity properties of the Euler operators $\beta_i$, as stated in Lemma \ref{lem:betacomm}. \end{proof}

\begin{prop}\label{prop:tensor} The algebra  $\DqN^\circ$ is isomorphic to the tensor product of the algebras $\left(\D_q(\C^1)^\circ \right)^{\ot N}$. \end{prop}

\begin{proof} Let $\widetilde{\D}_q(\C^N)^\circ$ be the algebra obtained from $\DqN^\circ$ by adjoining elements $\alpha_i$ such that $\alpha_i^2 = \beta_i$, and if $j > i$, then $ \alpha_i x_j = x_j \alpha_i  \ \text{and} \ \alpha_i \partial_j = \partial_j \alpha_i,$ while if $j \leq i$, then $ \alpha_i x_j = q x_j \alpha_i \ \text{and} \ \alpha_i \partial_j = q^{-1} \partial_j \alpha_i.$ We see that $\DqN^\circ$  embeds into $\widetilde{\D}_q(\C^N)^\circ$ as  the subalgebra generated by $x_i$, $y_i$, and $\beta_i\inv$. Set $w_i=x_i\alpha_i\inv, z_j=-q \partial_j \alpha_i\inv$, as elements of $\widetilde{\D}_q(\C^N)^\circ$. Since each $\alpha_i$ is invertible, there is an isomorphism between  $\DqN^\circ$ and the  subalgebra of $\widetilde{\D}_q(\C^N)^\circ$ generated by $z_i$, $w_i$, and $\beta_i\inv$. One shows by direct computation that the elements $w_i$ and $z_i$ satisfy: 
	$$w_j w_i = w_i w_j \qquad \text{\rm and} \qquad  z_j z_i = z_i z_j \qquad \text{for $j >i$},$$ 
	$$z_j w_i = w_i z_j \qquad \text{for $j \neq i$},$$
$$w_i z_i = q^{2} z_iw_i + (q^2-1).$$
Thus we have an isomorphism of algebras $$\widetilde{\D}_q(\C^N)^\circ \stackrel{\sim}{\longrightarrow} \left(\widetilde{\D}_q(\C^1)^\circ \right)^{\ot N}$$ taking $z_i$ to the $x$-variable of the $i$-th factor, $w_i$ to the $\partial$-variable of the $i$-th factor, and $\alpha_i$ to the $\alpha\inv$-variable of the $i$-th factor. It follows that this isomorphism restricts to an isomorphism between the subalgebra of $\widetilde{\D}_q(\C^N)^\circ$ generated by $z_i$, $w_i$, and $\beta_i\inv$ and the subalgebra  $\left({\D}_q(\C^1)^\circ \right)^{\ot N}$ of $ \left(\widetilde{\D}_q(\C^1)^\circ \right)^{\ot N}$. Since the former is isomorphic to $\DqN^\circ$, we obtain an isomorphism $ \DqN^\circ \stackrel{\sim}{\longrightarrow}  \left({\D}_q(\C^1)^\circ \right)^{\ot N}  $. We summarize this discussion in the following diagram:
\[\xymatrix{
 \DqN^\circ \ar[rr]^\sim \ar@{_{(}->}[d]  & & \langle z_i, w_i, \beta_i\inv \rangle \ar@{_{(}->}[d] \ar[rr]^\sim & &  \left({\D}_q(\C^1)^\circ \right)^{\ot N}  \ar@{_{(}->}[d]  \\
\widetilde{\D}_q(\C^N)^\circ \ar[rr]^\sim & & \widetilde{\D}_q(\C^N)^\circ \ar[rr]^\sim & & \left(\widetilde{\D}_q(\C^1)^\circ \right)^{\ot N}
} \] \end{proof}

\begin{theorem}\label{thm:DqNAzumaya} Suppose $q$ is a primitive $\ell$-th root of unity, where $\ell > 1$ is odd. The Azumaya locus of $\DqN$ is $\Spec(Z_N^\circ).$ In particular, $\DqN^\circ$ is Azumaya over $Z_N^\circ$, and its rank is $\ell^N \times \ell^N$. \end{theorem}

\begin{proof}[Proof of Theorem \ref{thm:DqNAzumaya}.] It is immediate from Proposition \ref{prop:center} that $\DqN$ is finitely generated and projective (in fact free) as a module over its center.   We show first that  $\Spec(Z_N^\circ)$ contains the Azumaya locus of $\DqN$. Fix $ \nu = (\nu_i, \check \nu_i)_{i=1}^N \in \Spec(Z_N)$. The fiber of $\DqN$ over $\nu$ is the algebra
$$\D_\nu:=\DqN \slash\DqN\paren{x_{i}^\ell -\nu_{i},\partial^{\ell}_{i}-\check \nu_{i}}.$$
 Suppose that $\nu \in \Spec(Z_N)$ is a closed point in the Azumaya locus, so $\D_\nu$ is a matrix algebra. The $q$-commutativity of the $\beta_i$ with the generators of $\DqN$ (Lemma \ref{lem:betacomm}, parts 3 and 4) implies that the image of each $\beta_i$ generates a nonzero 2-sided ideal in $\D_\nu$. As a matrix algebra, $\D_\nu$ has no nonzero proper ideals. We conclude that the image of each $\beta_i$ is invertible. Thus, $\nu$ belongs to $\Spec(Z_N^\circ)$. For the reverse direction, it follows from \cite[Theorem 3.12]{GanevQuantizationsmultiplicativehypertoric2018} that ${\D}_q(\C^1)^\circ$ is Azumaya over its center. The tensor product of matrix algebras is again a matrix algebra, and hence $\left({\D}_q(\C^1)^\circ \right)^{\ot N} $ is Azumaya over its center. The claim now follows from Proposition \ref{prop:tensor}. \end{proof}

\begin{rmk}
We note in passing that Theorem \ref{thm:DqNAzumaya} can be proved alternatively by exhibiting an explicit isomorphism between $\DqN^\circ$ and an Ore localization of the standard quantum torus on $2N$ generators, and then comparing the central embeddings.
\end{rmk}
	

\section{The reflection equation algebra for $\GL_2$}\label{sec:qmm}

\subsection{Basic definitions}

Fix a positive integer $n$. For  an $n \times n$ matrix $M$, set $M_1 = M \ot \id$ and $M_2 = \id \ot M$, where $\id = \id_n$ is the $n \times n$ identity matrix. Note that $M_1$ and $M_2$ are matrices of dimension $n^2 \times n^2$.   For $q \in \C^\times$,  define the following two matrices:
$$R = \begin{bmatrix}
q & 0 &0 &0  \\
0 & 1 &0 &0  \\
0 & q-q\inv &1 &0  \\
0 & 0 &0 & q 
\end{bmatrix}, \qquad R_{21} = \begin{bmatrix}
q & 0 &0 &0  \\
0 & 1 & q-q\inv &0  \\
0 & 0 &1 &0  \\
0 & 0 &0 &q
\end{bmatrix}  $$

\begin{definition}For $q \in \C^\times$, define an algebra $\O_q^+(\GL_2)$ as generated by four elements, $a$, $b$, $c$, and $d$, organized in a two by two matrix $L = \begin{bmatrix} a & b \\ c&  c \end{bmatrix}$, with  relations given by the entries of the matrix equation \begin{equation}\label{eq:refleq} R_{21}L_1RL_2 = L_2R_{21}L_1R.\end{equation}\end{definition}

\begin{lemma}\label{lem:oqgl2relns} We have the following:
	\begin{enumerate}
		\item 	The quantum determinant $\det_q = ad - q^2bc$  and quantum trace $\text{\rm tr}_q = a+ q^{-2}d$ are central elements of $\O_q^+(\GL_2)$. 
		\item For $n,m \geq 0$, the following relations hold:
		$$ d^n a^m = a^m d^n \qquad \qquad \qquad  d^n b^m = q^{2nm} b^m d^n \qquad \qquad \qquad d^n c^m = q^{-2nm} c^m d^n$$
		$$b^n a = a b^n + (q^{2(n-1)}-q^{-2}) b^n d \qquad \qquad  \qquad c^n a = a c^n + (q^{-2} - q^{2(n-1)}) d c^n$$
		\begin{align*} c^n b &= b c^n+ (q^{2(n-1)} - q^{-2}) adc^{n-1} + (q^{-2n} - 1) c^{n-1}d^2 \\ c b^n &= b^n c + (1- q^{-2n}) adb^{n-1} + \left(\frac{-q^{4n} + q^{4n-2} - q^{2(n-1)} + 1)}{q^{4}} \right) b^{n-1}d^2 \end{align*}
		\item The left ideal generated by $d$ coincides with the right ideal generated by $d$.
	\end{enumerate}
	 \end{lemma}

\begin{proof} The defining relations of $\O_q^+(\GL_2)$ can be written explicitly as:
	\begin{align*} 
	da &= ad &\qquad  db &= q^2 bd & \qquad dc & =q^{-2}cd \\
	cb &=  bc + (1 - q^{-2})(ad  - d^2)  &\qquad ba &= ab + (1- q^{-2}) b d  &\qquad ca &= ac + (q^{-2} - 1) dc 
	\end{align*}
From this description, the centrality of the quantum determinant and the quantum trace are straightforward verifications. The first three commutation relations are immediate; the remainder follow by induction. \end{proof}

Note that the defining relations imply that the element $d$ satisfies the Ore localization condition.

\begin{definition} Define $\O_q(\GL_2)$ as the localization of $\O_q^+(\GL_2)$ at the central element $\det_q$.  Define $\O_q(BB^-)$ to be the localization $\OqGL[d\inv]$ of $\OqGL$ at the element $d$.\end{definition}

\begin{rmk} The algebra $\O_q(\GL_2)$ is known as the reflection equation algebra, and the equation \ref{eq:refleq} is known as the reflection equation.  The matrix $R$ is the usual $R$-matrix on $\C^2$. Note that, when $q=1$, we obtain the classical commutative algebras. The algebra $\O_q(BB^-)$ is a quantization of the algebra of functions on the big Bruhat cell $BB^-$ of $\GL_2$, and is in fact isomorphic to the quantum group $\U_q(\gl_2)$, upon further adjoining a square root of $d$ and of $\detq$. See \cite{KlimykQuantumGroupsTheir1997} for further details.\end{rmk}


\subsection{The algebra $\Dq(\GL_2)$}
\begin{definition} For $q \in \C^\times$, define an algebra $\Dq^+(\GL_2)$ as generated by eight elements, $x_{11}, x_{12},$ $x_{21},$ $x_{22}$, $\partial_{11}$, $\partial_{12}$, $\partial_{21}$, and $\partial_{22}$ , organized in two by two matrices $$X = \begin{bmatrix} x_{11} & x_{12} \\ x_{21}&  x_{22}  \end{bmatrix}, \qquad D=  \begin{bmatrix}  \partial_{11} & \partial_{12}\\ \partial_{21} & \partial_{22}  \end{bmatrix} ,$$ with  relations given by the entries of the matrix equations $$R_{21}X_1RX_2 = X_2R_{21}X_1R, \qquad R_{21}\partial_1R\partial_2 = \partial_2R_{21}\partial_1R, \qquad R_{21}\partial_1RX_2 = X_2R_{21}\partial_1(R_{21})\inv.$$
\end{definition}

\begin{lemma}\label{lem:dqgl2det} The elements $\detq(X) := x_{11} x_{22} - q^2 x_{12}x_{21}$ and $\detq(D) := \partial_{11} \partial_{22} - q^2 \partial_{12} \partial_{21}$ satisfy:\begin{align*}
			\detq(X) x_{ij} =&x_{ij} \detq(X), \qquad & \partial_{ij} \detq(X) =& q^{-2} \detq(X) \partial_{ij}\\
			\detq(D) x_{ij} =&q^{-2} x_{ij} \detq(D), \qquad &  \partial_{ij} \detq(D) =& \detq(D)\partial_{ij} 
		\end{align*}	
	\end{lemma}

\begin{definition} Define $\Dq(\GL_2)$ as  the localization of $\Dq^+(\GL_2)$ at the element $\detq(X)\detq(D)$. \end{definition}

\begin{rmk} There are two embeddings of $\OqGL$ into $\Dq(\GL_2)$, one given by the assignment $ L \mapsto X$, and the other by $L \mapsto D.$ The images of these two maps together generate $\Dq(\GL_2)$, so in some sense the latter is a twisted tensor product of two copies of $\OqGL$. The cross relations are given by the formulas in the following lemma. \end{rmk}
		
\begin{lemma}\label{lem:dqgl2rels} For any $n,m \geq 1$, the following identities hold in $\Dq^+(\GL_2)$:
		$$\partial_{12}^n  x_{12}^m = q^{-2nm} x_{12}^m \partial_{12}^n, \qquad \partial_{22}^n  x_{12}^m = x_{12}^m \partial_{22}^n, \qquad \partial_{21}^n  x_{21}^m = q^{-2nm} x_{21}^m \partial_{21}^n,$$ $$ \partial_{21}^n  x_{22}^m = x_{22}^m \partial_{21}^n, \qquad \partial_{22}^n  x_{22}^m = q^{-2nm} x_{22}^m \partial_{22}^n$$
		\begin{align*}
			\partial_{21}^n  x_{12} =&  x_{12} \partial_{21}^n + q^{-2-2n}(-q^{2n} + 1) x_{22} \partial_{21}^{n-1}\partial_{22}, \qquad &\partial_{21} x_{12}^n  =& x_{12}^n \partial_{21} + q^{-4}(1-q^{2n} ) x_{12}^{n-1} x_{22} \partial_{22}\\
			\partial_{22}^n x_{21} =&  q^{-2n} x_{21} \partial_{22}^n + q^{-4n+2}(1-q^{2n} ) x_{22} \partial_{21} \partial_{22}^{n-1}, \quad & \partial_{22} x_{21}^n =& q^{-2n} x_{21}^n \partial_{22} + q^{-4n +2}(1-q^{2n} ) x_{21}^{n-1} x_{22} \partial_{21}\\
			\partial_{12}^n x_{22} =& q^{-2n} x_{22} \partial_{12}^n + q^{-2n}(1-q^{2n} ) x_{12} \partial_{12}^{n-1} \partial_{22} , \quad & \partial_{12} x_{22}^n =& q^{-2n} x_{22}^n \partial_{12} + q^{-2n}(1-q^{2n} ) x_{12} x_{22}^{n-1} \partial_{22} 
		\end{align*}
		\begin{align*}
			\partial_{12}^n x_{21} =& x_{21} \partial_{12}^n +  q^{-2}(1-q^{2n} ) x_{11} \partial_{12}^{n-1}\partial_{22} + q^{-2n}(q^{4n-2} - q^{2n} - q^{2n-2} + 
			1)x_{12} \partial_{11} \partial_{12}^{n-2} \partial_{22} \\
			&+q^{-2n}(q^{2n+2} - q^{2n} - q^2 + 1) x_{12} \partial_{12}^{n-1} \partial_{21} \\
			&+ q^{-2n-2}(-q^{6n-4} + q^{6n-6} - q^{4n-6} + q^{2n} + q^{2n-4} - 1) x_{12} \partial_{12}^{n-2} \partial_{22}^2\\
			&+(1-q^{-2n} )x_{22} \partial_{11} \partial_{12}^{n-1} + q^{-2n-2}(q^{4n} - q^{4n-2} + q^{2n-2} -1) x_{22} \partial_{12}^{n-1} \partial_{22}\\
			\\
			\partial_{12} x_{21}^n =&x_{21}^n \partial_{12} (q^{-2n} - 1) x_{11} x_{21}^{n-1}\partial_{22} + q^{-4n+2}(q^{4n-2} - q^{2n} - q^{2n-2} + 1) x_{11} x_{21}^{n-2} x_{22} \partial_{21}\\
			&+ q^{-2n}(q^{2n+2} - q^{2n} - q^2 + 1) x_{21}^{n-1} x_{12} \partial_{21} + (q^{-2n} -1)x_{21}^{n-1}x_{22}\partial_{11}\\
			&+ (1 - q^{-2n}) x_{21}^{n-1} x_{22} \partial_{22} 
			+ q^{-4n-2}(-q^{4n-2} + q^{2n} + q^{2n-2} - 1)x_{21}^{n-2} x_{22}^2 \partial_{21}\\
			\\
			\partial_{12}^n x_{11} =&  x_{11} \partial_{12}^n + (q^{-2n} - 1)x_{12} \partial_{11} \partial_{12}^{n-1} + q^{-2n-2}(q^{4n} - q^{4n-2} + q^{2n-2} - 1)x_{12} \partial_{12}^{n-1}\partial_{22}\\
			\partial_{21}^n x_{11} =& q^{-2n} x_{11} \partial_{21}^n + (-q^{-2n} + q^{-4n}) x_{21}\partial_{21}^{n-1}\partial_{22} + q^{-4n}(q^{4n} - q^{4n-2} -  q^2 + 1)x_{22} \partial_{21}^n\\
			\partial_{22}^n x_{11} =&  x_{11} \partial_{22}^n + q^{-2n+2}(1-q^{2n} ) x_{12} \partial_{21} \partial_{22}^{n-1}\\
			\\
			\partial_{11} x_{12}^n =& q^{-2n} x_{12}^n \partial_{11} + q^{-2n-2}(q^{2n+2} - q^{2n} - q^2 + 1)x_{12}^{n}\partial_{22} + q^{-2n-2}(-q^{2n} 
			+ 1)x_{12}^{n-1} x_{22} \partial_{12}\\
			\partial_{11}x_{21}^n =&  x_{21}^n \partial_{11}  + (q^{-2n} - 1)x_{11} x_{21}^{n-1} \partial_{21} + (1- q^{-2n}) x_{21}^{n-1}x_{22}\partial_{21}\\
			\partial_{11} x_{22}^n =&  (1-q^{2n} ) x_{12} x_{22}^{n-1}\partial_{21} + x_{22}^n \partial_{11}
		\end{align*}
\end{lemma}

\begin{proof} All claims are  easily verified by straightforward computations, some of which involve  inductive arguments.  \end{proof}


\subsection{Behavior at a root of unity}

\begin{prop}\label{prop:OqCenter} Suppose $q$ is a primitive $\ell$-th root of unity, for $\ell >1$ odd. 
\begin{enumerate}
	\item The $\ell$-th powers of $b$, $c$, and $d$ are central in $\OqGL$, and there is a unique central element $z$ such that $\text{\rm det}_q^\ell  = zd^\ell - b^\ell c^\ell.$ 
	
	\item The  following elements are central in $\Dq(\GL_2)$:
	$$x_{12}^\ell , \ x_{21}^\ell, \ x_{22}^\ell, \ \partial_{12}^\ell ,\  \partial_{21}^\ell,\  \partial_{22}^\ell,\  \detq(X)^\ell, \ \text{\rm and} \ \detq(D)^\ell.$$
	Moreover, there are uniquely defined central elements $v$ and $w$ in $\Dq(\GL_2)$ such that $$\detq(X)^\ell  = v x_{22}^\ell - x_{12}^\ell x_{21}^\ell, \quad \text{\rm and} \quad \detq(D)^\ell  = w \partial_{22}^\ell - \partial_{12}^\ell \partial_{21}^\ell .$$
\end{enumerate}	
\end{prop}

\begin{proof} The identities listed in Lemma \ref{lem:oqgl2relns} imply that $b^\ell$, $c^\ell$, and $d^\ell$ are central in $\OqGL$. Next, observe that all monomials that appear in $\detq^\ell$ are of the form $a^k d^k b^m c^m d^{2p}$ for $k+m+p = \ell$. Let $f_{k,m,p}$ be the coefficient of this monomial. Since $\detq$ is central, we have the identity $\detq^\ell b = b \detq^\ell$. The left hand side expands as:
$$\sum f_{k,m,p} a^k d^k b^m c^m d^{2p} b = b\left(\sum f_{k,m,p} q^{2(k+2p)} a^k d^k b^m c^m d^{2p} + \dots\right)$$
where all terms that appear in `\dots' are of the form $a^{k'} d^{k'} b^{m'} c^{m'} d^{2p'}$ where the triple $(k',m',p')$ is not equal to the triple $(k,m,p)$. The justification of this expression follows from the commutation relations that $b$ has with the other generators (see Lemma \ref{lem:oqgl2relns} above). Comparing with the right-hand side of $b \detq^\ell$, we see that $q^{2(k+2p)}$ must be equal to 1. In other words, $\ell$ divides $k+2p$. It follows that $d^\ell$ divides  $\detq^\ell + b^\ell c^\ell$, and hence $\detq^\ell = zd^\ell - b^\ell c^\ell$. The fact that $\detq$, $d^\ell$, $b^\ell$ and $c^\ell$ are all central implies that $z$ is as well.
	
 The claim that $x_{12}^\ell$, $x_{21}^\ell$, $x_{22}^\ell$, $\partial_{12}^\ell$,  $\partial_{21}^\ell$, $\partial_{22}^\ell$, $\detq(X)^\ell$, and $\detq(D)^\ell$ are central in  $\Dq(\GL_2)$ follows by direct inspection of the identities in Lemmas \ref{lem:dqgl2det} and \ref{lem:dqgl2rels}, noting also that the commutation relations among the $x$'s and the $\partial$'s are each the same as the commutation relations for $\OqGL$ (given in Lemma \ref{lem:oqgl2relns}). Let $v$ and $w$ be the images of $z \in \OqGL$ under the embeddings $L \mapsto X$ and $L \mapsto D$, respectively. The fact  $v$ is central follows easily.   \end{proof} 

\begin{definition} We define the following algebras:
	\begin{enumerate}
		\item Let $\OqGLell$ be the subalgebra of $\OqGL$ generated by the elements  $z, b^\ell, c^\ell$, and $d^\ell$. 
		\item Let  $\Dq^{(\ell)}(\GL_2)$ be the subalgebra of $\Dq(\GL_2)$ generated by the elements  $v$, 	$x_{12}^\ell$, $x_{21}^\ell$, $x_{22}^\ell$, $w$, $\partial_{12}^\ell$, $\partial_{21}^\ell$, and  $\partial_{22}^\ell$.  
	\end{enumerate}\end{definition}

\begin{cor} \label{cor:GL2ell} 	We have:
	\begin{enumerate}
		\item 	The subalgebra $\OqGLell$ is central in $\OqGL$, and is isomorphic to the coordinate algebra $\O(\GL_2)$. Similarly, the subalgebra $\OqGLell[(d^\ell)\inv]$ is central in $\O_q(B B^-)$, and is isomorphic to the coordinate algebra $\O(BB^-)$ of the big Bruhat cell $B B^-$ in $\GL_2$.  
		\item The subalgebra $\Dq^{(\ell)}(\GL_2)$ is central in $\Dq(\GL_2)$, and is isomorphic to the coordinate algebra $\O(\GL_2 \times \GL_2)$. 
	\end{enumerate}
\end{cor}

\begin{rmk} We note that the element $a^\ell$ in $\OqGL$ is not central, and hence neither are the elements $x_{11}^\ell$ and $\partial_{11}^\ell$ in $\Dq(\GL_2)$. The leading term of $z$ is $a^\ell$, and it is the unique central element with that property, so it can be seen as a correction to the naive powering operation being incompatible with the quantum Frobenius homomorphism of Lusztig.  \end{rmk}

\begin{example} In the case $\ell = 3$, direct computation gives that $z= a^3 + 3abc + 3qbcd$. For $\ell = 5$, one finds similarly, $$z = a^5 + 5a^3bc + 5(2q^3 - q)a^2bcd + 5ab^2c^2 + 5(2q^3  + 3q^2 + 4q + 2)abcd^2 + 5q^3b^2c^2d - 5(q^3+ 2q^2+ q)bcd^3.$$
It would be interesting to compute the general form of $z$. 
\end{example}


\subsection{A quantum moment map for $\Dq(\C^2)$} We now connect the reflection equation algebra to the algebra $\DqN$ of Section \ref{sec:dqn} in the case $N=2$. 

As explained in Section \ref{sec:dqn}, the element $\beta_2 = 1 + x_1 \partial_1 + x_2 \partial_2 \in \Dq(\C^2)$ is the total grading operator, and let $\Dq(\C^2)^{\det}$ denote the localization $\Dq(\C^2)[\beta_2\inv]$ of $\Dq(\C^2)$ at this element.  

\begin{prop}\label{prop:dq2mm} There is a well-defined algebra homomorphism $$\mu_q: \OqGL \rightarrow \Dq(\C^2)^{\det}$$ 
	\begin{align*} a &\mapsto 1+ \partial_2 x_2 &\qquad  b &\mapsto  \partial_2 x_1 \\
	c &\mapsto   \partial_1 x_2 &\qquad d &\mapsto 1+ x_1 \partial_1. \end{align*}
	Moreover, $\mu_q$ takes the quantum determinant $\detq $ to the total grading operator $\beta_2$. \end{prop}

\begin{proof} The reflection equation algebra relations are easily verified directly. Alternatively, one can appeal to \cite[Definition-Proposition 7.11]{JordanQuantizedmultiplicativequiver2014}. \end{proof}

Let $Z= Z_2$ denote the center of $\Dq(\C^2)$. Recall from Proposition \ref{prop:betaeqns} that $\beta_2^\ell = 1+ x_1^\ell \partial_1^\ell + x_2^\ell \partial_2^\ell$ belongs to $Z$, and set $Z^{\det}$ to be the localization $Z[ (\beta_2^\ell)\inv]$. Thus,  $Z^{\det}$ is the algebra of functions on the subvariety $$ (T^*\C^2)^{\det} = \{(a,b, \xi, \zeta) \in \C^2 \times \C^2 \simeq \C^2 \times (\C^2)^* \simeq T^*\C^2 \ : \ 1 + a\xi + b \zeta \neq 0\} $$ of $T^*\C^2$. As is well-known, the action of $\GL_2$ on $\C^2$ extends to a Hamiltonian action of $\GL_2$ on $T^*\C^2$ which admits a group-valued moment map $$\tilde \mu : (T^*\C^2)^{\det} \rightarrow \GL_2 \ \qquad \qquad \tilde \mu(a,b, \xi, \zeta) =  \begin{bmatrix}  1 + a\xi  & a\zeta \\  b\xi & 1 + b\zeta  \end{bmatrix}$$
We obtain a pullback homomorphism: $\tilde \mu^\# : \O(\GL_2) \rightarrow \O((T^*\C^2)^{\text{\rm det}}).$

\begin{theorem}\label{thm:dq2mmcenter} Suppose $q$ is a primitive $\ell$-th root of unity, for $\ell >1$ odd.
\begin{enumerate}
\item\label{thm:dq2mmcenter:muq}  Then $\mu_q$ restricts to an algebra homomorphism $\mu_q^{(\ell)} : \OqGLell \rightarrow Z^{\det}$ given by:
	\begin{align*} 
	z & \mapsto 1+ x_2^\ell \partial_2^\ell &   b^\ell & \mapsto x_1^\ell \partial_2^\ell \\
	c^\ell & \mapsto x_2^\ell \partial_1^\ell &  d^\ell & \mapsto 1+ x_1^\ell \partial_1^\ell \end{align*}
	Moreover, $\detq^\ell$ maps to the element $1+ x_1^\ell \partial_1^\ell + x_2^\ell \partial_2^\ell$. 

\item Under the identification of $\OqGLell$ with $\O(\GL_2)$ and $Z^{\det}$ with $\O((T^* \C^2)^{\det})$, the map $\mu_q^{(\ell)}$ matches with $\tilde \mu^\#$. Thus, the following diagram commutes: \[ \xymatrix{  \O_q(\GL_2 ) \ar[r]^{\mu_q} & \Dq(\C^2)^{\det} \\
\O(\GL_2 ) \ar@{^{(}->}[u] \ar[r]^{\tilde \mu^\#  \ \ } & \O( (T^* \C^2 )^{\text{\rm det}}) \ar@{^{(}->}[u] }. \]

\item The  Azumaya locus $(T^*\C^2)^\circ$ is the preimage of the big Bruhat cell under the group-valued moment map $\tilde \mu$. Thus, the following diagram commutes:
$$ \xymatrix{ \O_q(BB^-) \ar[r]^{\mu_q}  & \D_q(\C^2)^\circ  \\ \O(BB)  \ar[u] \ar[r]^{\tilde \mu^\#} & \O(T^* \C^2)^\circ \ar[u]   } $$
\end{enumerate}
\end{theorem}

\begin{proof} To prove the first claim, first note that the fact that $x_1$ and $\partial_2$ commute up to a power of $q$ implies that $b^\ell$ maps to $x_1^\ell \partial_2^\ell$. Similarly for $c^\ell$.  Proposition \ref{prop:betaeqns} implies that $d^\ell$ maps to $1+ x_1^\ell \partial_1^\ell $. The same proposition also implies that $\beta_2^\ell = 1+ x_1^\ell \partial_1^\ell + x_2^\ell \partial_2^\ell$. From this, and the relation $\detq^\ell = zd^\ell - b^\ell c^\ell$, one verifies the claim for the image of $z$. The second assertion is immediate from definitions. For the last statement, observe that, upon localization of $d$ and $\mu_q(d) = 1 + x_1 \partial_1$, the map $\mu_q$ descends to an algebra homomorphism
$$\OqGL[d\inv] \rightarrow D_q(2)^\circ = \D_q(\C^2)[(1 + x_1 \partial_1 + x_2 \partial_2)\inv, (1+ x_2 \partial_2)\inv].$$
This map restricts to  $$\O(BB^-) =  \OqGLell[(d^\ell)\inv] \rightarrow Z^\circ$$ The source of the above map can be identified with the algebra of functions on the big Bruhat cell $B  B^-$, while the target is precisely the Azumaya locus of $D_q(2)$, namely $(T^*\C^2)^\circ$. \end{proof}

\begin{rmk} The map $\mu_q$ is in fact a `quantum moment map' because it satisfies a certain equivariance condition for the actions of the quantum group for $\gl_2$. We do not describe these actions explicitly here. See \cite{JordanQuantizedmultiplicativequiver2014, GanevquantumFrobeniuscharacter2019} for more details. \end{rmk}


\subsection{A quantum moment map for $\Dq(\GL_2)$} In this last section we state without proof an explicit reformulation of results due to Varagnolo and Vasserot \cite{VaragnoloDoubleaffineHecke2010} for $\Dq(\GL_N)$, using the discussion in the preceding section.  The group-commutator form of the quantum moment map is from \cite{JordanQuantizedmultiplicativequiver2014}; we refer therein for a discussion of the comparison to the moment map from \cite{VaragnoloDoubleaffineHecke2010}.   These constructions are generalized to arbitrary groups $G$, and to arbitrary multiplicative quiver varieties in \cite{GanevquantumFrobeniuscharacter2019}. 

Let $\GL_2$ act on its `double' $D(G) =\GL_2 \times \GL_2$ by diagonal conjugation: $g (A,B) = (gAg\inv, gBg\inv)$. As is well-known \cite{AlekseevLiegroupvalued1998}, there is a group-valued moment map $$\tilde \phi : \GL_2 \times \GL_2 \rightarrow \GL_2 \ \qquad \qquad \tilde \phi(A,B) =  AB\inv A\inv B.$$
We obtain a pullback homomorphism: $\tilde \phi^\# : \O(\GL_2) \rightarrow \O(\GL_2 \times \GL_2).$ A quantum version of this map is given in the following proposition, which holds for any $q \in \C^\times$. 

\begin{prop}\label{prop:dqgl2mm}(\cite[Definition-Proposition 7.20]{JordanQuantizedmultiplicativequiver2014}, \cite[Proposition 1.8.3]{VaragnoloDoubleaffineHecke2010}) There is a quantum moment map $\phi_q: \OqGL \rightarrow \Dq(\GL_2)$ defined by $$ L \mapsto DX\inv D \inv X.$$ \end{prop}

In other words, the map $\phi_q$ sends $a \in \OqGL$ to the element of $\Dq(\GL_2)$ appearing in the top left entry of the product $DX\inv D \inv X$, and so on for the other generators. 

Now assume $q$ is a primitive $\ell$-th root of unity, for $\ell >1$ odd. By Corollary \ref{cor:GL2ell}, the coordinate algebra $\O(\GL_2 \times \GL_2)$ is a central subalgebra of $\Dq(\GL_2)$, and we can regard the latter as defining a coherent sheaf of algebras over $\GL_2 \times \GL_2$. Let $(\GL_2 \times \GL_2)^\text{\rm Az} \subseteq \GL_2 \times\GL_2 $ denote the Azumaya locus of $\Dq(\GL_2)$. 

\begin{definition} Let   $L^{(\ell)} $ be the following matrix with values in $\O_q^{(\ell)}(\GL_2)$:
	$$L^{(\ell)} = \begin{bmatrix}	z & b^\ell \\ c^\ell & d^\ell	\end{bmatrix}.$$
	Similarly, we define matrices $X^{(\ell)}$ and $D^{(\ell)}$ with entries in $\Dq^{(\ell)}(\GL_2)$:
		$$X^{(\ell)} = \begin{bmatrix}	v & x_{12}^\ell \\ x_{21}^\ell & x_{22}^\ell	\end{bmatrix}, \qquad  \qquad D^{(\ell)} = \begin{bmatrix}	w & \partial_{12}^\ell \\ \partial_{21}^\ell & \partial_{22}^\ell \end{bmatrix}.$$
\end{definition}

Thus, the matrix $L^{(\ell)}$ encodes the generators of $\O_q^{(\ell)}(\GL_2)$, while $X^{(\ell)}$ and $D^{(\ell)}$ encode the generators of $\Dq^{(\ell)}(\GL_2).$ 

\begin{prop}\label{prop:dqgl2mmcenter} (\cite[Theorem 4.5]{GanevquantumFrobeniuscharacter2019}) Suppose $q$ is a primitive $\ell$-root of unity, for $\ell >1$ odd. 
\begin{enumerate}
\item Then $\phi_q$ restricts to an algebra homomorphism $\phi_q^{(\ell)} : \O_q^{(\ell)}(\GL_2) \rightarrow \Dq^{(\ell)}(\GL_2)$ given by $$L^{(\ell)} \mapsto D^{(\ell)} (X^{(\ell)})\inv (D^{(\ell)})\inv X^{(\ell)}$$ 

\item  Under the identification of $\OqGLell$ with $\O(\GL_2)$ and $\D_q^{(\ell)}(\GL_2)$ with $\O(\GL_2 \times \GL_2)$, the map $\phi_q^{(\ell)}$ matches with $\tilde \phi^\#$. Thus, the following diagram commutes: \[ \xymatrix{  \O_q(\GL_2 ) \ar[rrr]^{\phi_q} & & &  \D_q(\GL_2) \\
		\O(\GL_2 ) \ar@{^{(}->}[u] \ar[rrr]^{\tilde \phi^\#  \ \ } & &  & \O( \GL_2 \times \GL_2) \ar@{^{(}->}[u] }. \]

\item The Azumaya locus of $\Dq(\GL_2)$ over $\GL_2 \times \GL_2$ is the inverse image of the big Bruhat cell: 
$$(\GL_2 \times \GL_2)^\text{\rm Az}  = \tilde\phi\inv(BB^-).$$
\end{enumerate}
\end{prop}


\bibliography{Quiver-compute}

\medskip

\noindent \rule{7cm}{0.4pt}

\medskip

\noindent{\small

\noindent  {\sc E-mail:} \href{mailto:cooney.maths@gmail.com}{\tt cooney.maths@gmail.com}

\medskip

\noindent {\sc IST Austria, Klosterneuburg, Austria}.   {\sc E-mail:} \href{mailto:iordan.ganev@ist.ac.at}{\tt iordan.ganev@ist.ac.at}

\medskip

\noindent {\sc School of Mathematics, University of Edinburgh, Edinburgh, UK}.    {\sc E-mail:} \href{mailto:D.Jordan@ed.ac.uk}{\tt D.Jordan@ed.ac.uk}}

\vfill

\end{document}